\documentclass[11pt,leqno,english]{amsproc}

\usepackage{times,eucal,txfonts,dsfont,mathptmx,newtxmath,newtxtext}

\makeatletter
\def\msection{\@startsection{section} 
{1} 
{0pt} 
{-1ex plus -.1ex minus -0.9ex} 
{-.9ex plus -.2ex} 
{\bfseries} 
}
\def\msubsection{\@startsection{subsection} 
{2} 
{0pt} 
{-1ex plus -.1ex minus -0.2ex} 
{-.9ex plus -.2ex} 
{\normalfont} 
} 
\makeatother 

\usepackage[dvipsnames]{xcolor} 
\usepackage{xparse}
\usepackage{xr-hyper}
\usepackage[linktocpage=true,colorlinks=true,hyperindex,citecolor=teal,linkcolor=blue]{hyperref}

\usepackage[text={6.5in,9in},centering]{geometry}

\newcommand{\spec}{\mathfrak{P}(L)}

\newcommand{\mx}{\mathfrak{M}(L)}
\newcommand{\jr}{\mathsf{j}_L}
\newcommand{\z}{\mathfrak{B}(L)}
\newcommand{\cm}{\mathfrak{C}(L)}
\newcommand{\p}{^{\perp}}
\newcommand{\pp}{^{\perp\perp}}

\DeclareMathOperator{\cz}{\mathfrak{c}}

\numberwithin{equation}{section}
\newtheorem{theorem}[equation]{Theorem}
\newtheorem{proposition}[equation]{Proposition}
\newtheorem{lemma}[equation]{Lemma}
\newtheorem{corollary}[equation]{Corollary}
\theoremstyle{definition}
\newtheorem{definition}[equation]{Definition}
\newtheorem{example}[equation]{Example}
\newtheorem{remark}[equation]{Remark}

\begin{document}
	
\author{Amartya Goswami  
}

\address{[1] Department of Mathematics and Applied Mathematics, University of Johannesburg.
[2] National Institute for Theoretical and Computational Sciences (NITheCS).
}

\email{agoswami@uj.ac.za}

\author{Themba Dube}

\address{[1] Department of Mathematical Sciences, University of South Africa.
[2] National Institute for Theoretical and Computational Sciences (NITheCS).
}
\email{dubeta@unisa.ac.za}

\title{Revisiting  Baer elements}




%

\maketitle

\begin{abstract}
The objective of this paper is to extend certain properties observed in $d$-ideals of rings and $d$-elements of frames to Baer elements in multiplicative lattices introduced in D. D. Anderson, C. Jayaram, and P. A. Phiri, Baer lattices, \textit{Acta Sci. Math. (Szeged)},
59 (1994), 61--74. Additionally, we present results concerning these elements that have not been addressed in the study of $d$-ideals of rings. Furthermore, we introduce Baer closures and explore Baer maximal, prime, semiprime, and meet-irreducible elements.
\end{abstract}
\smallskip
  
\msection{Introduction}\label{intro}

Baer-ideals were first introduced in \cite{Spe97} for commutative Baer rings, where they were defined as the kernels of  Baer ring homomorphisms. Subsequently, in \cite{Jay84}, the study of Baer ideals was extended to semiprime commutative rings. In \cite{Con88}, the concept of such ideals was generalized to any commutative ring and termed ``$B$-ideals''. The term ``$d$-ideal'' in rings, not necessarily commutative, was introduced in \cite{Mas89}, with the terminology adapted from the series of papers \cite{HP80, HP80', Pag81} focusing on $d$-ideals of Riesz spaces. Furthermore, in \cite{AKA99}, a study was conducted on some properties of $d$-ideals in rings $C(X)$, referred to as $z^0$-ideals. Also, $d$-ideals have been examined for reduced rings under the same name in \cite{AKA00}.

The initial effort to extend $d$-ideals to ``$d$-elements'' in lattices was undertaken in \cite{MZ03} for frames, once again borrowing terminology from \cite{HP80}. Subsequent research on $d$-elements can be found in \cite{DS19} and \cite{Bha19}, where the second focused on maximal $d$-elements. It was in \cite{AJP94} where $d$-elements were further extended to encompass more general types of multiplicative lattices, not necessarily frames, and labelled Baer elements.

Since the introduction of Baer elements in \cite{AJP94} for multiplicative lattices, several studies have been conducted on $d$-ideals in rings and $d$-elements in frames, as can be seen above. All of these studies may warrant scrutiny for further extensions to multiplicative lattices in the sense of \cite{AJP94}. Furthermore, there are several properties of $d$-ideals of rings studied in earlier literature that have not been fully captured in their extended form in \cite{AJP94}. Our modest aim in this paper is to extend some of those results to Baer elements. Moreover, we shall include some aspects of Baer elements that, to the best of our knowledge, have not even been studied for $d$-ideals in rings per se.

Let us now briefly describe the content of this paper. 
In Section \ref{prlm}, we gather the multiplicative lattice theoretic machinery needed throughout this paper. All of the material presented is standard, with references including \cite{Dil36, War37, WD39, Dil62, And74, And76}.

In Section \ref{zid}, we present various equivalent characterizations and basic properties of Baer elements. We provide a necessary and sufficient condition for the closedness of Baer elements under finite products (Theorem \ref{cup}), characterize lattices where every Baer element is a $z$-element (Theorem \ref{beiz}), characterize domains (Theorem \ref{cdbe}), and discuss several relations between prime, maximal, and Baer elements (Proposition \ref{pmb}). We also give a sufficient condition for the contraction of Baer elements under a multiplicative lattice homomorphism (Theorem \ref{bep1}).

In Section \ref{bcl}, we introduce Baer closures and discuss some properties of them (Proposition \ref{lclk}). We provide some equivalent criteria on closedness of Baer elements under joins (Proposition \ref{cujo}). We view Baer closures as nuclei of various types, and apply them in studying some distinguished classes of Baer elements, namely maximal, prime, semiprime, and meet-irreducible.

\smallskip

\msection{Preliminaries} \label{prlm}
 
From \cite{Dil62}, recall that a \emph{multiplicative lattice} is a complete lattice $(L, \leqslant, 0, 1)$ endowed with an associative, commutative multiplication (denoted by $\cdot$), which distributes over arbitrary joins and has $1$ as multiplicative identity. Note that  in the sense of \cite{Mul86}, a multiplicative lattice is a commutative, unital quantale.
For brevity, we shall write $xy$ for $x\cdot y$ and  $x^n$ for $x\cdot   \cdots \cdot   x$ (repeated $n$ times). We say $x$ is \emph{strictly below} $y$, denoted by $x<y$, whenever $x\leqslant y$ and $x\neq y$. Recall that a map $\phi\colon L \to L'$ is said to be a \emph{multiplicative lattice homomorphism} if $\phi$ preserves $\leqslant$, binary joins, binary meets, and multiplication.

For our later purposes, in the following lemma, we collect some elementary properties of multiplication of multiplicative lattices.

\begin{lemma}\label{bip}
In a multiplicative lattice $L$, the following hold.
\begin{enumerate}
	
\item\label{pxyx} $xy\leqslant x$, for all $x$, $y\in L.$
	
\item \label{mul}
$x  y\leqslant x\wedge y$, for all $x, y\in L$.
		
\item $x  0=0$, for all $x\in L$.
		
\item\label{mon} If $x\leqslant y$, then $x  z\leqslant y  z,$\, for all $x$, $y$, $z\in L$.
		
\item\label{monj} If $x\leqslant y$ and $u\leqslant v$, then $x  u\leqslant y v$, for all $x,$ $y,$ $u,$ $v\in L$.
\end{enumerate}
\end{lemma}


Let us now recall a few definitions. An element $a$ in $L$ is called \emph{nilpotent} if we have $a^n=0$ for some $n\in \mathds{N}^+$ (where $\mathds{N}^+$ is the set of positive natural numbers ), and $L$ is said to be \emph{reduced} if $0$ is the only nilpotent element in $L$. An element $c$ in $L$ is called \emph{compact}, if whenever  $\{x_{\lambda}\}_{\lambda \in \Lambda}\subseteq L$, and $c\leqslant \bigvee_{\lambda \in \Lambda} x_{\lambda}$, we have $c\leqslant \bigvee_{i=1}^n x_{\lambda_i}$ for some $n\in \mathds{N}^+$, and  $L$ is said to be \emph{compactly generated} if every element of
$L$ is the join of compact elements in $L$. We shall denote by $\cm$ the set of compact elements in $L$.

Let us pause with the recollection of our definitions in order to declare our underlying assumptions on lattices. For the rest of the paper, we will only work with  reduced, compactly generated multiplicative lattices where $1$ is compact and
every finite product of compact elements is compact.

An element $x$ in $L$ is called \emph{proper} if ${x<1}$. A proper element $m$ in $L$ is said to be \emph{maximal} if, whenever $x\in L$, $m \leqslant x$, and $x< 1$, we have $m = x$.   We shall denote by $\mx$ the set of all maximal elements in $L$. The following well-known result confirms that $\mx$  is nonempty.

\begin{lemma}\label{flm}
In a multiplicative lattice $L$, if the top element $1$ in $L$ is compact, then for every proper element $a$ in $L$, there exists an $m\in \mx$ such that $a\leqslant m$.
\end{lemma}

A proper element $p$ in $L$ is called \emph{prime} if, whenever $x$, $y\in L$, $x y\leqslant p$, we have  either $x\leqslant p$ or $y\leqslant p$. We shall denote by $\spec$ the set of prime elements in $L$. An element $p$ in $L$ is called a \emph{minimal prime} if there is no
prime element $q$ in $L$ with $q < p$. An element $q$ in  $L$ is said to be \emph{semiprime} if, whenever $x\in L$, and $x^2\leqslant q$, we have  $x\leqslant q$. 

Let us recall that an element $s$ in  $L$ is called \emph{meet-irreducible} if, whenever $x$, $y\in L$, and $x\wedge y \leqslant s$, we have  either $x\leqslant s$ or $y\leqslant s$. An element $x$ in $L$ is said to be a \emph{zero divisor} if there exists a non-zero element $y$ in $L$ such that $xy=0$, and $L$ is said to be a \emph{domain} if $L$ does not have any non-zero zero divisor.

According to  \cite{AAJ95}, the \emph{radical  of an element} $x$ in $L$ is defined as
\[\sqrt{x}:=\bigvee\left\{y\in L\mid \;y\;\text{is compact},\;y^n\leqslant x,\;\text{for some}\; n\in \mathds{N}^+\right\}.\]
It is then shown that
\begin{align}
\label{empp}\tag{$\dagger$}\sqrt{x}&=\bigwedge\left\{p\in \spec\mid p\;\text{is prime over}\; x\right\}\\
&=\bigwedge\left\{p\in \spec\mid p\;\text{is minimal prime over}\; x\right\}\nonumber,	
\end{align}
and  $x$ is called a \emph{radical element} whenever $x=\sqrt{x}$. The \emph{Jacobson radical} of a multiplicative lattice $L$ is defined as \[\jr:=\bigwedge\left\{m\mid m\in \mx \right\},\]  and $L$ is said to be \emph{semisimple} if $\jr=0$. 
Recall that  for any two elements $a$ and  $b$ in $L$,  the \emph{residual} of $a$ by $b$ is defined as \[(a:b):=\bigvee\left\{ x\in L\mid xb\leqslant a\right\},\] and  the \emph{annihilator} of an element $a$ in $L$ is defined as \[a\p:=\bigvee\left\{ x\in L\mid xa=0\right\}.\]
To denote the annihilator of the annihilator $a\p$  of an element $a$, instead of $(a\p)\p$, we shall use the notation $a\pp$. 
Let us record some well-known properties of annihilators of elements in multiplicative lattices that will be useful in sequel.

\begin{lemma}\label{pae}
Suppose $L$ is a multiplicative lattice and $a$, $b$, $c$ are elements in $L$. Then the following hold.
\begin{enumerate}
\item\label{mdp} If $a\leqslant b$, then $b\p\leqslant a\p$.
		
\item\label{cpp} $a\leqslant a\pp$.
		
\item\label{tap} $a^{\perp \perp\perp}=a\p.$
		
\item\label{dam} $a\pp\wedge b\pp=(ab)\pp.$
		
\item \label{mip}
If $a\p\leqslant b\p$, then $(ac)\p \leqslant (bc)\p.$
\end{enumerate}	
\end{lemma}

We shall denote by $\mathfrak{A}(L)$ the set of all annihilators of $L$. That is,
\[
\mathfrak{A}(L)= \{x^\perp\mid x\in L\} =\{y^{\perp\perp} \mid y\in L\}.
\]
It is proper to comment that, as in frames, with the partial order inherited from $L$, $\mathfrak{A}(L)$ is a Boolean algebra.

\smallskip

\msection{Baer elements}
\label{zid}

According to \cite[Definition 1]{AJP94}, an element $a$ in a multiplicative lattice $L$ is called a \emph{Baer element}  if, whenever $c\in \cm$, and $c\leqslant a$, we have  $c\pp\leqslant a$. We shall denote by $\z$ the set of all Baer elements in $L$. 

In the next proposition, we gather some elementary results on Baer elements. 

\begin{proposition}
\label{epbe}
In a multiplicative lattice $L$, then  the following hold.
\begin{enumerate}
\item \label{ald}
An element $b$ in $L$ is a Baer element if and only if, whenever $u$, $v\in \cm$, $u\p=v\p$, and $u\leqslant b$, we have $v\leqslant b$.

\item\label{ebre} Every Baer element in $L$ is a radical element.

\item\label{mtp} If $b$ is a Baer element in $L$, then so is every minimal prime element above $b$.

\item\label{bepb} Every Baer element in $L$ is the meet of all prime Baer elements above it.

\item\label{zao} Both $0$ and $1$ in $L$ are Baer elements.

\item\label{ijam} If $\{x_{i}\}_{i\in I}$ are Baer elements in $L$, then so is $\bigwedge_{i\in I} x_i$.
\end{enumerate}
\end{proposition}

\begin{proof}
(1)--(4) See \cite{AJP94}.

(5) For any elements $u$, $v\in \cm$, let us suppose that $u\p=v\p$ and $u\leqslant 0$. This implies that $u=0$ and hence $0\p=1=v\p$. Therefore  $v=0$. The proof of $1$ is a Baer element is trivial.

(6) Let us suppose $a$, $b\in \cm$ with  $a\p = b\p$ and $b\leqslant \bigwedge_{\lambda\in \Lambda}x_{\lambda}$. Since $\bigwedge_{\lambda\in \Lambda}x_{\lambda}\leqslant x_{\lambda}$ and  $x_{\lambda}\in \z$ for all $\lambda\in \Lambda$,  we have $a\leqslant x_{\lambda}$ for all $\lambda\in \Lambda$, and hence $a\leqslant \bigwedge_{\lambda\in \Lambda}x_{\lambda},$ implying that $\bigwedge_{\lambda\in \Lambda}x_{\lambda}\in \z.$
\end{proof}

In Proposition \ref{epbe}(\ref{ald}), we have seen an alternative definition of a Baer element. In the next result we shall provide another one.

\begin{proposition}
An element $x$ in a multiplicative lattice $L$ is a Baer element if and only if, whenever  $u$, $v\in \cm$, $u\p\leqslant v\p$, and $u\leqslant x$, we have  $v\leqslant x$.
\end{proposition}

\begin{proof}
Let $x$ be a Baer element in $L$. Suppose that $u$, $v\in \cm$, $u\p\leqslant v\p$, and $u\leqslant x$.  Since $x$ is a Baer element and $u$ compact, it follows that $u\pp\leqslant x$. Using this  and (\ref{mdp}), (\ref{cpp}) from Lemma \ref{pae},  we obtain
\[v\leqslant v\pp \leqslant u\pp\leqslant x.\]
Thanks to Proposition \ref{epbe}(\ref{ald}), the converse is obvious.
\end{proof}

In the context of frames, there is another well-known definition of a 
$d$-element (see \cite[Definition and Remarks 5.1.]{MZ03}), which can also be extended to a Baer element in a multiplicative lattice. Furthermore, this definition coincide to one of the (equivalent) definitions mentioned earlier. We summarize all these facts in the following proposition.

\begin{proposition}
The following are equivalent for an element $d$ in a multiplicative lattice $L$.
\begin{enumerate}
\item $d$ is a Baer element.

\item $d=\bigvee \{ c\pp\mid c\leqslant d, c\in \cm\}$.
\end{enumerate}
\end{proposition}

\begin{proof}
(1)$\Rightarrow$(2): Since $L$ is compactly generated, 
$d=\bigvee\{c\in \cm\mid c\leqslant d\}.$
Now if $d$ is a Baer element, then $v\pp\leqslant d$, whenever $v\in \cm$ and $v\leqslant d$.Thus, since $w\leqslant w\pp$ for all $w\in L$,
\[ d=\bigvee\{c\in \cm\mid c\leqslant d\}\leqslant \bigvee \{ c\pp\mid c\leqslant d, c\in \cm\}\leqslant d, \]
which proves that $d=\bigvee \{ c\pp\mid c\leqslant d, c\in \cm\}$.

(2)$\Rightarrow$(1): Let us suppose that $d=\bigvee \{ c\pp\mid c\leqslant d, c\in \cm\}$. Then, whenever $v\in \cm$ and $v\leqslant d$, we have $v\pp \leqslant d$, which implies $d$ is a Baer element.
\end{proof}

\begin{remark}\label{edbe}
Let us collect all the  characterizations of Baer elements that we have  so far. In a multiplicative lattice $L$, the following are equivalent for an element $b\in L$:
\begin{enumerate}
\item[$\bullet$]
$b\in \z$.

\item[$\bullet$]
Whenever $u$, $v\in \cm$, $u\p=v\p$, and $u\leqslant b$, we have  $v\leqslant b$.

\item[$\bullet$]
Whenever $u$, $v\in \cm$, $u\p\leqslant v\p$, and $u\leqslant b$, we have  $v\leqslant b$.

\item[$\bullet$] 
$b=\bigvee \{ c\pp\mid c\leqslant b, c\in \cm\}$.
\end{enumerate}
Later on, in Proposition \ref{lclk}(\ref{altd}), we shall see one more (equivalent) characterization of a Baer element.
\end{remark}

In Proposition \ref{epbe}(\ref{ijam}), we have shown  that Baer elements are closed under arbitrary meets. The next result will give us a partial converse of that and  extend \cite[Proposition 2.14]{ST19}.

\begin{proposition}\label{mibe}
Let $\{a_i\}_{i=1}^n$ be a set of elements in a multiplicative lattice $L$ with the property that $a_i\vee a_j=1$ whenever $i\neq j$. If $\bigwedge_{i=1}^n a_i$ is a Baer element in $L$, then each $a_i$ $(1\leqslant i \leqslant n)$ is also a Baer element in $L$.
\end{proposition}

\begin{proof}
Suppose $u$ and $v$ are compact elements in $L$, $u\p \leqslant v\p$ and $u\leqslant a_j$, for some $j\in \{1, \ldots, n\}.$ It is easy to see that
$ a_j\vee \bigwedge_{k=1,\, k\neq j}^n a_k =1,$
from which we obtain 
\[va_j \vee \bigwedge_{k=1,\, k\neq j}^n va_k =v.\]
Observe that $u\p \leqslant v\p$ implies $(ua_k)\p\leqslant (va_k)\p$, for any $k\neq j$. Since $ua_k\leqslant \bigwedge_{k=1}^na_k$ and $\bigwedge_{k=1}^n a_k\in \z$, this implies that \[\bigwedge_{k=1,\, k\neq j}^n va_k \leqslant \bigwedge_{i=1}^n a_k,\]
and hence $\bigwedge_{k=1,\, k\neq j}^n va_k \leqslant a_j$. Obviously $va_j\leqslant a_j$. Therefore $v\leqslant a_j$, proving that $a_j\in \z$.
\end{proof}

We shall now inquire when Baer elements are closed under finite products, and the following theorem provides us with a necessary and sufficient condition for that.

\begin{theorem}\label{cup}
The product of two Baer elements in a multiplicative lattice $L$ is a Baer element if and only if $c\pp c\pp=c\pp$ for all compact elements $c$ in $L$.
\end{theorem}

\begin{proof}
Let us suppose that $a$ and $b$ are two Baer elements in $L$. Suppose $c\leqslant ab$ for some compact element $c$ in $L$. This implies that $c\leqslant a$ and $c\leqslant b$. Since $a$ and $b$ are Baer elements, we have $c\pp\leqslant a$ and $c\pp \leqslant b$. Applying the hypothesis and Lemma \ref{bip}(\ref{monj}), we obtain
\[c\pp=c\pp c\pp\leqslant ab,\]
implying $ab$ is a Baer element.
Conversely, let us suppose that $c$ be a compact element in $L$. Since $c\leqslant c\pp$, we have $c^2\leqslant (c\pp)^2$. Moreover, $(c\pp)^2$ is a Baer element by our assumption. Using all these and Lemma \ref{pae}(\ref{dam}), we have the desired identity from the following:
\[c\pp=c\pp \wedge c\pp=(cc)\pp\leqslant (c\pp)^2=c\pp c\pp\leqslant c\pp. \qedhere\] 
\end{proof}

\begin{definition}
We shall say a multiplicative lattice $L$ is a $B$-\emph{multiplicative lattice} if $(c\pp)^2=c\pp$ holds for all compact elements $c$ in $L$. 
\end{definition}

Our next goal is to obtain some examples of Baer elements using residuals.  

\begin{proposition}\label{rpb}
In a multiplicative lattice $L$, the following hold.
\begin{enumerate}
	
\item\label{rbe} If $a$ is a Baer element in $L$, then so is $(a:b)$ for any element $b$ in $L$.
 
\item\label{anb} For every element $a$ in $L$, the element $a\p$ is a Baer element.

\item If   $b$, $\{b_{i}\}_{i \in I}$, $c$ are elements in $L$  and   $a$, $\{a_{j}\}_{j \in J}$ are Baer elements in $L$, then    $((a : b) : c),$ $(a : b c),$ $((a : c) : b),$ $ (\bigwedge_{j\in J} a_{j} : b),$ $ \bigwedge_{j\in J}(a_{j}: b),$ $ (a : \bigvee_{i\in I}b_{i}),$ and $ \bigwedge_{i\in I}(a : b_{i})$ are  Baer elements in $L$.
\end{enumerate}
\end{proposition}

\begin{proof}
(1) Let us suppose $u$ and $v$ are compact elements in $L$, $u\p\leqslant v\p$, and $u\leqslant (a:b).$ This implies that
\[bu\leqslant (a:b)b\leqslant a.\]
Since $(bu)\p\leqslant (bv)\p$ by Lemma \ref{pae}(\ref{mip}) and since $a\in \z$, we have $bv\leqslant a$, which implies that $v\leqslant (a:b)$.

(2) Although the proof follows from (\ref{rbe}) and Proposition \ref{epbe}(\ref{zao}), however, we give another proof. Suppose $v\in \cm$ and $v\leqslant a\p$. Applying (\ref{mdp}) and (\ref{tap}) from Lemma \ref{pae}, we have $v\pp\leqslant a\p$. Hence $a\p\in \z$.

(3) Follows from (\ref{rbe}).
\end{proof}

We next wish to show that certain types of multiplicative lattice homomorphisms ``contract'' Baer elements to Baer elements. Since we are going to impose certain conditions on the homomorphisms, we fist show by a simple example that the conditions we will impose are not so stringent as to make them isomorphisms. 

\begin{example}
Let $\mathbf 2$ be the two-element chain $\{0,1\}$ and $\mathbf 3 =\{0,m,1\}$ be the three-element chain; both viewed as multiplicative lattices with product given by $\wedge$. Let $\phi\colon \mathbf 2\to\mathbf 3$ be the unique homomorphism. Then $\phi$ has the following properties:
\begin{itemize}
\item[$\bullet$]
$\phi$ maps compact elements to compact element.

\item[$\bullet$]
$\phi$ is injective, but not an isomorphism.

\item[$\bullet$]
For any compact $c$ in the domain of $\phi$, $\phi(c^\perp)=\phi(c)^\perp$.
\end{itemize} 
\end{example}

\begin{definition}
We call a  multiplicative lattice homomorphism $\phi\colon L\to M$ {\em strong} if it is injective, it maps compact elements to compact elements, and it commutes with annihilation of compact elements in the sense that for any compact $c\in L$, we have  $\phi(c^\perp)=\phi(c)^\perp$. 	
\end{definition}

\begin{theorem}\label{bep1}
If $\phi\colon L \to M$ is a strong homomorphism, then $\phi^{-1}[\mathfrak B(M)]\subseteq \mathfrak B(L)$. 
\end{theorem}

\begin{proof}
Let $x\in \phi^{-1}[\mathfrak B(M)]$, and consider any compact elements $c$ and $d$ in $L$ with $c^\perp=d^\perp$ and $c\leqslant x$. We aim to show that $d\leqslant x$, which, by Proposition \ref{epbe}(\ref{ald}), will prove that $x$ is a Baer element. Since $\phi$ is strong, $\phi(c)$ and $\phi(d)$ are compact element of $M$. Moreoever, $\phi(c) \leqslant \phi(x)$. Since $x\in \phi^{-1}[\mathfrak B(M)]$, we have $\phi(x)\in\mathfrak B(M)$. It follows, therefore, that $\phi(d)\leqslant \phi(x)$, whence $d\leqslant x$ since $\phi$ is injective. Thus, $x\in\mathfrak B(L)$, as desired.  
\end{proof}

In terms of similarity, one of the closest classes of elements to Baer elements is that of $z$-elements. To define a $z$-element, we need  the confirmation of existence of maximal elements in  multiplicative lattices, which follows from Lemma \ref{flm}.
For an $a$ in $L$, let us define \[\mathcal{M}_a:=\left\{ m\in \mx \mid a\leqslant m\right\}\quad \text{and} \quad \mathsf{m}_a:= {\bigwedge}\mathcal{M}_a.\]
According to \cite{MC19}, an element $x$ in  $L$ is called a \emph{z-element} if, whenever $a$, $b\in L$,  $\mathcal{M}_a=\mathcal{M}_b$, and $b\leqslant x $, we have $a\leqslant x$. 

\begin{remark}
It can be shown (see \cite[Lemma 2.10]{MC19}) that the following three statements are equivalent in a multiplicative lattice $L$ (\textit{cf}. Remark \ref{edbe}):

\begin{enumerate}
\item[$\bullet$]An element $x$ is a $z$-element in $L$.
	
\item[$\bullet$]Whenever $a$, $b\in L$, $\mathcal{M}_a\supseteq \mathcal{M}_b$, and $b\leqslant x $, we have $a\leqslant x$. 
	
\item[$\bullet$]Whenever $a$, $b\in L$,  $\mathsf{m}_b\leqslant \mathsf{m}_a$, and $a\leqslant x$, we have $b\leqslant x$.  
\end{enumerate}
\end{remark}

In the next theorem, we shall characterize those multiplicative lattices in which every Baer element is a $z$-element. This result is well-known in rings (\textit{e.g.}, see \cite[Proposition 2.9]{ST19}).  

\begin{theorem}
\label{beiz}
In a multiplicative lattice $L$, every Baer element is a $z$-element if and only if $L$ is semisimple.
\end{theorem}

\begin{proof}
Let $x$ be a Baer element in a semisimple multiplicative lattice $L$. Let us suppose that $\mathsf{m}_b\leqslant \mathsf{m}_a$ and $a\leqslant x$. We claim that $a\p \leqslant b\p$. For $c\in \cm$ with $ca=0$ implies that $c\leqslant a\p$. Now from \cite[Lemma 2.6]{MC19} and the hypothesis, it follows that 
\[\mathsf{m}_{bc}\leqslant \mathsf{m}_{ac}=\mathsf{m}_{0}=\jr=0.\]
Therefore $bc=0$, and hence $c\leqslant b\p$. Since $x\in \z$, we have $b\leqslant x$, and  hence $x$ is a $z$-element in $L$.
Conversely, let us suppose that every Baer element in $L$ is a $z$-element. From Proposition \ref{epbe}(\ref{zao}) it follows that $0$ is a $z$-element. Now if $a\in \jr$, then we have 
$\mathsf{m}_a=\mathsf{m}_0=0,$
which implies that $a=0$.
\end{proof}

\begin{corollary}
In an algebraic frame, every $d$-element is a $z$-element.
\end{corollary}

\begin{remark}
For the ring $\mathcal{R}L$ of continuous functions over a frame $L$,  the notions of  $d$-ideals and $z$-ideals are introduced in \cite[Definition 4.10]{Dub09},  and it is not hard to see that every $d$-ideal in $\mathcal{R}L$ is a $z$-ideal. Therefore, our Theorem \ref{beiz} also extends this result to multiplicative lattices.
\end{remark}

We shall now obtain a characterization of domains by the absence of non-trivial Baer elements in them, and the next result extends \cite[Lemma 2.11]{ST19}.

\begin{theorem}
\label{cdbe}
A multiplicative lattice $L$ is a domain if and only if $L$ has no non-zero Baer element.
\end{theorem}

\begin{proof}
Let us suppose that $a$ is a Baer element in a domain $L$. This implies $a\pp=0$, and since  $a\leqslant a\pp$ by Lemma \ref{pae}(\ref{cpp}), it follows that $a=0$. Conversely, suppose $a$ is an element in $L$ such that $ax=0$ for some $x\in L$. This implies that $x\leqslant a\p .$ By  Proposition \ref{epbe}(\ref{anb}), $a\p$ is a Baer element and by hypothesis $a\p=0$. Therefore $x=0$, implying that $L$ is a domain.
\end{proof}

Our next proposition is on establishing some relations between Baer, prime, and maximal elements in multiplicative lattices, and first three of them extend \cite[Proposition 3.1]{ST19}.

\begin{proposition}
\label{pmb}
In a multiplicative lattice $L$, the following hold.
\begin{enumerate}
\item If $x\in L$, $p\in \spec$, and  $x\wedge p\in \z$, then either $x\in \z$ or $p\in \z$.

\item If $p$, $q\in \spec$  which do not belong to a chain and $p\wedge q\in \z$, then both $p$, $q\in \z$.

\item If $x\in L$, $m\in \mx$  with $x\nleqslant m$, and $x\wedge m\in \z$, then $x$, $m\in \z$.

\item If $p$ is a prime element in $L$, then either $p$ is a Baer element or the maximal
Baer elements that are below $p$ are prime Baer elements.
\end{enumerate}
\end{proposition}

\begin{proof}
(1) Let us consider the case: $x\leqslant p$. Then by hypothesis, $x\wedge p =x\in \z$, and we are done. So, let $x\nleqslant p$. Then there exists an $y\in L$ with $y\leqslant x$ and $y\nleqslant p$. Suppose $u$ and $v$ are two compact elements in $L$ with $u\p \leqslant v\p$ and $u\leqslant p$. From Lemma \ref{pae}(\ref{mip}), we have $(uy)\p \leqslant (vy)\p$. Since $uy\leqslant x\wedge p$ and by hypothesis, $x\wedge p\in \z$, we have $vy\leqslant x\wedge p$, and hence $vy\leqslant p$. Since $p$ is prime and $y\nleqslant p$, it follows that $v\leqslant p$. Therefore $p$ is a Baer element in $L$.

(2) Let us suppose that $x\leqslant q$ and $x\nleqslant p$. Suppose $u$, $v\in \cm$, $u\p \leqslant v\p$ and $u\leqslant p$. By Lemma \ref{pae}(\ref{mip}), $(ux)\p \leqslant (vx)\p$. Since $ux\leqslant p\wedge q$ and $p\wedge q \in \z$, we have \[vx\leqslant p\wedge q\leqslant p.\] Since $p$ is prime and $x\nleqslant p$, we must have $v\leqslant p$ implying that $p$ is a Baer element in $L$. Similarly, we can show that $q$ is also a Baer element in $L$.

(3) Since $x\nleqslant m$, this means $x\vee m=1$, and hence by Proposition \ref{mibe} we have the desired claim. 

(4) Let us define a set \[S:=\left\{x\in \z \mid x\leqslant p\right\}.\] Since $0\in S$ by Proposition \ref{epbe}(\ref{zao}), the set $S$ must be nonempty. Therefore, by Zorn's lemma, $S$ has a maximal element, say $m$. Now, $m=p$ if and only if $p$ is a prime Baer element. If $m<p$, then there exists a prime element  $q$ which is minimal with respect to above $m$ and below $p$. By Proposition \ref{epbe}(\ref{mtp}), $q$ is a Baer element, and hence $q\neq p$. This implies either $q=m$, hence $m$ is a prime element, or $m< q$,  contradicting maximality of $m$.
\end{proof}

\smallskip

\msection{Baer closures and their applications}\label{bcl}

We shall now introduce a closure  on a multiplicative lattice that will assign a unique Baer element to each element of the lattice and will give another alternative definition of a Baer element. Furthermore, using these closures, we shall study some distinguished classes of Baer elements in due course.

\begin{definition}\label{dbl}
The \emph{Baer closure} on a multiplicative lattice $L$ is the map $\cz\colon L \to L$ defined as
\[\cz(a):=\bigwedge\left\{ x\in \z \mid a\leqslant x\right\}.\]
\end{definition}

In the next proposition, we  shall  establish some essential properties of Baer closures.

\begin{proposition}\label{lclk}
Suppose $a$ and $b$ are elements in a  multiplicative lattice $L$. Then the following hold.
\begin{enumerate}
		
\item\label{iclk} $\cz(a)$ is the smallest Baer element such that $a\leqslant \cz(a)$.
		
\item\label{altd}   $\cz(a)=a$ if and only if $a$ is a Baer element.
		
\item \label{ckr}
$ \cz(a) =1$ if and only if $a=1$.
		
\item \label{ssz}
$\cz(0)=0$.
		
\item\label{ijcl} If $a\leqslant b$, then $\cz(a)\leqslant \cz(b).$
		
\item\label{clcl} $\cz(\cz(a))=\cz(a).$
		
\item \label{clsq}  $\sqrt{a}\leqslant \cz(a)$.
		
\item\label{craca} $\cz(a)=\cz(\sqrt{a})$.
		
\item\label{clan} $\cz(a^n)=\cz(a)$, for any $n\in \mathds{N}^+$.
		
\item\label{clab} $\cz(a) \vee \cz(b)\leqslant \cz( a\vee b )=\cz(\cz(a)\vee\cz(b)).$	
		
\item\label{tccl}  $\cz(ab)=\cz(a\wedge b)=\cz(a)\wedge \cz(b)$.
		
\item\label{clijk}If $L$ is a $B$-multiplicative lattice, then $\cz(ab)=\cz(a)\cz(b)$.
\end{enumerate}
\end{proposition}

\begin{proof}
(1) From Definition \ref{dbl} it follows that $a\leqslant \cz(a)$. By Proposition \ref{epbe}(\ref{ijam}) we know that $\cz(a)$ is a Baer element in $L$. Now suppose $y$ is a Baer element with the property $a\leqslant y$. Then from Definition \ref{dbl}, it implies that $\cz(a)\leqslant y$.
	
(2) If $a$ is a Baer element, then by Definition \ref{dbl}, we have $\cz(a)=a$. Conversely, if the identity holds, then $a$ is a Baer element  by (\ref{iclk}).
	
(3) Since $1$ is a Baer element by Proposition \ref{epbe}(\ref{zao}), the claim $\cz(1)=1$ follows by (\ref{altd}). The converse is obvious.
	
(4) Follows from Proposition \ref{epbe}(\ref{zao}) and (\ref{altd}).
	
(5) Let us suppose that $x\in \z$ with $b\leqslant x$. Then $a\leqslant x$ and hence $\cz(a)\leqslant x$. Since $x$ was an arbitrary Baer element with the property $b\leqslant x$, we must have $\cz(a)\leqslant \cz(b).$
	
(6) By (\ref{iclk}), we have $\cz(a)\leqslant \cz(\cz(a)).$ Now  suppose that $x\in \z$ with $\cz(a)\leqslant x$. Since by definition, $\cz(\cz(a))$ is the infimum of all such $x\in \z$, we must have $\cz(\cz(a))\leqslant \cz(a).$
	
(7) Let us suppose that $\mathrm{Spec}_B(L)$ denotes the set of all prime Baer elements of $L$. Now by applying (\ref{iclk}), (\ref{empp}), and Proposition \ref{epbe}(\ref{bepb}), we have
\begin{align*}
\sqrt{a}&=\bigwedge \{ p\in \spec \mid a\leqslant p\}\\&\leqslant \bigwedge \{ p\in \mathrm{Spec}_B(L) \mid a\leqslant p\}\\&\leqslant \bigwedge \{ p\in \mathrm{Spec}_B(L) \mid \cz(a)\leqslant p\}=\cz(a). 	
\end{align*}
	
(8) Applying (\ref{ijcl}), (\ref{clsq}), and (\ref{clcl}), we have $\cz(\sqrt{a})\leqslant \cz(a)$. On the other hand, applying (\ref{ijcl}) on $a\leqslant \sqrt{a},$ we get $\cz(a)\leqslant \cz(\sqrt{a}).$
	
(9) Since $a^n\leqslant a$, by (\ref{ijcl}), we have $\cz(a^n)\leqslant \cz(a).$ Let us now suppose $k\in \cm$ and $k\leqslant \cz(a).$ This implies that $k\leqslant x$ for all $x\in \z$ with $a\leqslant x$. Since $x\in \z$ and $a\p\leqslant (a^n)\p$, we have $a^n\leqslant x$. Therefore $k\leqslant \cz(a^n),$ proving that $\cz(a)\leqslant \cz(a^n).$
	
(10) Applying (\ref{iclk}) and (\ref{ijcl}), we obtain $\cz(a\vee b)\leqslant \cz(\cz(a)\vee \cz(b)),$ whereas applying (\ref{ijcl}), we have $\cz(a)\vee \cz(b)\leqslant \cz(a\vee b)$, and  applying (\ref{clcl}) on it gives \[\cz(\cz(a)\vee \cz(b))\leqslant \cz(\cz(a\vee b))=\cz(a\vee b).\]
	
(11) Since by Lemma \ref{bip}(\ref{pxyx}), we have $ab\leqslant a$ and $ab\leqslant b$, applying Lemma \ref{bip}(\ref{mul}), Proposition \ref{epbe}(\ref{ijam}), and (\ref{ijcl}), we obtain
\[\cz(ab)\leqslant \cz(a\wedge b)\leqslant \cz(a)\wedge \cz(b)\in \z.\]
Therefore, by (\ref{iclk}), to obtain the desired equalities, it suffices to show that $\cz(a)\wedge \cz(b)$ is the smallest Baer element that is above $ab$. Let us suppose that $x$ is a Baer element in $L$ with $ab\leqslant x$. We shall show that \[\cz(a)\wedge \cz(b)\leqslant x.\]
By $\mathrm{Min}_L(x)$, let us denote the set of all minimal prime elements in $L$ that are above $x$. By Proposition \ref{epbe}(\ref{mtp}), if $p\in \mathrm{Min}_L(x)$, then $p$ is a Baer element in $L$. By Proposition \ref{epbe}(\ref{bepb}), we have \[x=\bigwedge\left\{p\mid  p\in \mathrm{Min}_L(x) \right\}.\] Now if $p\in \mathrm{Min}_L(x)$, then $ab\leqslant p$, and hence, either $a\leqslant p$ or $b\leqslant p$. Since $p$ is a Baer element, this implies that $\cz(a)\wedge \cz(b)\leqslant p$. From this, it follows that $\cz(a)\wedge \cz(b)\leqslant x$, as required.
	
(12) The proof is similar to (\ref{tccl}).
\end{proof}

\begin{remark}
In the context of frames, another closure operation (see \cite[Definition and Remarks 5.1.]{MZ03}) is well-known for $d$-elements (that is, Baer elements in our context). 
For an element $x$ in a frame $L$, the \emph{closure} $d$ \emph{of} $x$ is defined as
\[d(x):=\bigvee \{ c\pp\mid  c\leqslant x, c\in \cm\}.\]
In this paper, we have not considered the above closure operator because in general, joins of $d$-elements are not $d$-elements, and we want to have our closure $\cz(\cdot)$ be such that $\cz(x)$ is the smallest $d$-element that is above $x$ (see Proposition \ref{lclk}(\ref{iclk})).
\end{remark}

In Section \ref{zid},  we have studied closedness properties of Baer elements under meets and products. The following proposition gives some equivalent criteria on closedness of Baer elements under joins. We are not aware of an explicit characterization of multiplicative lattices in which Baer elements are closed under joins.
\begin{proposition}\label{cujo}
In a multiplicative lattice $L$, the following are equivalent.
\begin{enumerate}
\item\label{cuj} If $a$ and $b$ are Baer elements in $L$, then so is $a\vee b$.
		
\item Whenever $a$, $b\in L$, we have $\cz(a\vee b)=\cz(a)\vee \cz(b).$
		
\item\label{ajp} If $\{a_{\lambda}\}_{\lambda\in \Lambda}\subseteq \z$, then $\bigvee_{\lambda\in \Lambda} a_{\lambda}\in \z.$
		
\item If $\{a_{\lambda}\}_{\lambda\in \Lambda}\subseteq L$, then $\cz\left(  \bigvee_{\lambda\in \Lambda} a_{\lambda}\right)=\bigvee_{\lambda\in \Lambda} \cz(a_{\lambda})$.
\end{enumerate}
\end{proposition}

\begin{proof}
We shall only show (1)$\Rightarrow$(3). The rest of the implications are straightforward. Suppose $u$, $v\in \cm$,  $u\p \leqslant v\p$, and $u\leqslant \bigvee_{\lambda\in \Lambda} a_{\lambda}$. Since $L$ is compactly generated, there exists a finite subset $\{ \lambda_1, \ldots, \lambda_n\}$ of $\Lambda$ such that $u\leqslant \bigvee_{i=1}^n a_{\lambda_i}$. Since by hypothesis $\bigvee_{i=1}^n a_{\lambda_i}$ is a Baer element, we have \[v\leqslant \bigvee_{i=1}^n a_{\lambda_i}\leqslant \bigvee_{\lambda\in \Lambda} a_{\lambda},\] implying that $\bigvee_{\lambda\in \Lambda} a_{\lambda}\in \z.$
\end{proof}

\begin{remark}
The equivalence of (\ref{cuj}) and (\ref{ajp}) in Proposition \ref{cujo},  extends the corresponsing result on $d$-ideals in rings (see \cite[Remark 3.2]{Dub19}).
\end{remark}

We now view Baer closures as  nuclei of various types. Recall from \cite[p.\,218]{NR88} that a map $\phi\colon L \to L$ is called a \emph{quantic nucleus} if, whenever $a$, $b\in L$, we have 
$a\leqslant \phi(a)$,
$\phi(\phi(a))=\phi(a)$, and
$\phi(ab)=\phi(a)\phi(b)$.

\begin{proposition}
If $L$ is a $B$-multiplicative lattice, then every Baer closure is a quantic nucleus.
\end{proposition}

\begin{proof}
Follows from (\ref{iclk}), (\ref{clcl}), and (\ref{clijk}) of Proposition \ref{lclk}.
\end{proof}

According to \cite[Definition 1]{Sim78}, the map $\phi$ is called a \emph{nucleus} if it satisfies 
$a\leqslant \phi(a)$,
$\phi(\phi(a))=\phi(a)$, and
$\phi(a\wedge b)=\phi(a)\wedge \phi(b)$, for all $a$, $b$ in $L$. By  \cite[Definition 1.1]{BH85}, the map $\phi$ is said to be a \emph{multiplicative nucleus} if it satisfies $\phi(a)=1$ if and only if $a=1$ and $\phi(ab)=\phi(a\wedge b)=\phi(a)\wedge \phi(b)$, for all $a$, $b\in L$.

\begin{proposition}
Every Baer closure on a multiplicative lattice is a  nucleus and a multiplicative nucleus.
\end{proposition}

\begin{proof}
The first claim follows from (\ref{iclk}), (\ref{clcl}), and (\ref{tccl}) of Proposition \ref{lclk}, whereas the second follows from (\ref{ckr}) and (\ref{tccl}) of Proposition \ref{lclk}.
\end{proof}

Here we shall gather a few results on Baer closures. The proofs of these statements are either identical or similar to those presented in \cite{NR88} and \cite{BH85}.

\begin{proposition}
If $L$ is a $B$-multiplicative lattice, then every Baer closure as a quantic nucleus satisfies the following identities:
\[\cz(ab)=\cz(a\cz(b))=\cz(\cz(a)b)=\cz(\cz(a)\cz(b)),\]
for all $a,$ $b\in L$.
\end{proposition}
Defining the join in $\z$  by $\bigvee'_{i\in I} a_i:=\cz\left( \bigvee_{i\in I} a_i\right),$ we obtain the following (\textit{cf}. \cite[Theorem 2.1]{NR88}).

\begin{theorem}
If $\cz$ is a quantic nucleus, then $\z$ is a multiplicative lattice via $a\odot b =  \cz(ab);$
and $\cz\colon L \to \z$ is a multiplicative lattice homomorphism. Moreover, every surjective multiplicative lattice
homomorphism arises (up to
isomorphism) in this manner.
\end{theorem}

Since $\cz$ satisfies the property: $\cz(ab)=\cz(a)\wedge \cz(b)$ (see Proposition \ref{lclk}(\ref{tccl})), it follows that $\z$ is a frame with $\odot := \wedge$, and $\cz$ is called a \emph{localic nucleus} (see \cite[p.\,219]{NR88} and \cite{Ban94}).
Since $\cz$ is a multiplicative nucleus, in fact, we get more  (\textit{cf}. \cite[Lemma 1.2]{BH85}).

\begin{theorem}
$\z$ is a compact frame.
\end{theorem}

It is well-known that to check the primeness of an element in an algebraic frame, we need only to verify the prime condition with respect to compact elements. With the applications of Baer closures, we shall now show that in multiplicative lattices, Baer elements, which need to be shown as maximal, prime, semiprime, or meet-irreducible elements, enjoy similar privileges with respect to Baer elements.

\begin{proposition}
Suppose $L$ is a multiplicative lattice. A Baer element $m$ in $L$ is maximal if and only if $m$ is maximal among all Baer elements in $L$.
\end{proposition}

\begin{proof}
The `only if' part is obvious. Therefore, suppose that  $m$ is a maximal element among all Baer elements in $L$ and let $m< a< 1,$ for some $a\notin \z .$ We need to show that $m$ is a maximal element in $L$. Now we have
\[m=\cz(m)< a < \cz(a)< \cz(1) =1,
\]
where the first and last equalities follow respectively from Lemma \ref{lclk}(\ref{altd}) and Lemma \ref{lclk}(\ref{ckr}), and the middle strict inclusion follows from  Lemma \ref{lclk}(\ref{altd}) and the assumption that $a\notin \z$.
Since $m$ is a maximal element in $\z$ and since $\cz(a)\in \z$, we have a contradiction. Thus, $m$ is a maximal element in $L$.
\end{proof}

The following corollary extends the second part of \cite[Theorem 3.4]{Bha19}.

\begin{corollary}
Every  maximal Baer element in a multiplicative lattice is  prime.
\end{corollary}

We shall now aim to obtain a similar result for  prime Baer elements and semirpime Baer elements, but for $B$-multiplicative lattices. It is not known to us whether this restriction on lattices are also necessary.

\begin{proposition}\label{pebe}
Suppose $L$ is a $B$-multiplicative lattice. Then the following hold.
\begin{enumerate}
\item A Baer element $p$ in $L$ is prime if and only if, whenever $x$, $y\in \z$, and $xy\leqslant p$,  we have either $x\leqslant p$ or $y\leqslant p$.

\item A Baer element $q$ in $L$ is semiprime if and only if, whenever  $x\in \z$,  and $x^2\leqslant q$, we have $x\leqslant q$.
\end{enumerate}
\end{proposition}

\begin{proof}
(1) Let us suppose that $p$ is a Baer element satisfying the hypothesis. Now suppose $x$, $y\in L$ and $xy\leqslant p$.  Then we obtain
\[\cz(x) \cz(y) \leqslant \cz(xy) \leqslant \cz(p))=p,\]
where the first inequality follows from Lemma \ref{lclk}(\ref{clijk}), whereas the second inequality holds due to Lemma \ref{lclk}(\ref{ijcl}), and finally, the equality is due to the fact that $p$ is a Baer element (see Lemma \ref{lclk}(\ref{altd})).
By hypothesis, this implies $x\leqslant  \cz(x)\leqslant p$ or $y \leqslant\cz(y)\leqslant p.$ Hence $p$ is a prime element in $L$. Once again, the `only if' part is obvious.

(2) Similar to (1).
\end{proof}

We conclude the paper with  a similar result as above, but for meet-irreducible elements. Thanks to Proposition \ref{lclk}(\ref{tccl}), however, we do not need any further restriction (as in Proposition \ref{pebe}) on our lattices. 

\begin{proposition}
A Baer element $s$ in a multiplicative lattice $L$ is meet-irreducible if and only if, whenever $x$, $y\in \z$, and $x\wedge y \leqslant s$, we have $x\leqslant s$ or $y\leqslant s$.
\end{proposition}

\begin{proof}
Let us suppose that $s$ is a Baer element in $L$ satisfying the hypothesis. Let $b$, $b'\in L$ and  $b\wedge b'\leqslant s.$ Applying (\ref{tccl}) and  (\ref{ijcl}) from Proposition \ref{lclk}, we obtain
\[\cz(b)\wedge \cz(b')=\cz(b\wedge b')\leqslant \cz(s)=s.\]
By hypothesis, this implies that  $ \cz(b)\leqslant s$ or $ \cz(b')\leqslant s$. Finally, by Proposition \ref{lclk}(\ref{iclk}), we  have the desired claim. The proof of the converse is trivial.
\end{proof}

\smallskip
\section*{Declarations}
The authors have no funding or competing interests to declare that are relevant to the content of this article. Also, this article does not involve any data.


\end{document}